\documentclass[twocolumns,10pt]{autart}

\usepackage{amsmath}
\usepackage{mathtools}
\usepackage{amssymb}
\usepackage{here}
\usepackage{mathrsfs}
\usepackage{enumerate}
\usepackage{multirow}
\usepackage{cite}
\usepackage{color}
\usepackage{dsfont}

\renewcommand{\theenumi}{\alph{enumi}}

\newtheorem{theorem}{Theorem}
\newtheorem{corollary}[theorem]{Corollary}
\newtheorem{proposition}[theorem]{Proposition}
\newtheorem{lemma}[theorem]{Lemma}
\newtheorem{define}[theorem]{Definition}
\newtheorem{example}[theorem]{Example}

\newtheorem{remark}[theorem]{Remark}

\DeclareMathOperator{\He}{He}
\DeclareMathOperator*{\diag}{diag}

\DeclareMathOperator{\eps}{\varepsilon}

\def\E{\mathbb{E}}

\def\vec{\textnormal{vec}}

\def\tablepos{h}

\newenvironment{proof}{{\it Proof :~}}{\hfill$\diamondsuit$\\}

\begin{document}

\begin{frontmatter}

\title{Stability analysis and stabilization of stochastic linear impulsive, switched and sampled-data systems under dwell-time constraints}

\author{Corentin Briat}\ead{briatc@bsse.ethz.ch,corentin@briat.info}\ead[url]{http://www.briat.info}

\address{Department of Biosystems Science and Engineering, ETH--Z\"{u}rich, Switzerland.}


\begin{abstract}
Impulsive systems are a very flexible class of systems that can be used to represent switched and sampled-data systems. We propose to extend here the previously obtained results on deterministic impulsive systems to the stochastic setting. The concepts of mean-square stability and dwell-times are utilized in order to formulate relevant stability conditions for such systems. These conditions are formulated as convex clock-dependent linear matrix inequality conditions that are applicable to robust analysis and control design, and are verifiable using discretization or sum of squares techniques. Stability conditions under various dwell-time conditions are obtained and non-conservatively turned into state-feedback stabilization conditions. The results are finally applied to the analysis and control of stochastic sampled-data systems. Several comparative examples demonstrate the accuracy and the tractability of the approach.
\end{abstract}
\end{frontmatter}

\section{Introduction}
\sloppy

Impulsive systems \cite{Michel:08,Goebel:09} arise in many applications such as such as ecology \cite{Yu:06b,Verriest:09d}, epidemiology \cite{Briat:09h} and sampled-data systems/control \cite{Ye:98,Naghshtabrizi:08,Briat:13d}. Lyapunov-based methods and dwell-time notions \cite{Hespanha:04} can be used to establish various stability concepts and conditions. Various concepts of dwell-time have been considered over the past decades. Important examples are the minimum dwell-time \cite{Morse:96,Hespanha:04,Geromel:06b,Briat:11l}, the average dwell-time \cite{Hespanha:04,Hespanha:08}, the persistent dwell-time \cite{Hespanha:04,Goebel:09,Zhang:15a}, the maximum dwell-time \cite{Briat:11l,Briat:12h,Briat:13d}, the ranged-dwell-time \cite{Briat:11l,Briat:12h,Briat:13d} and the mode-dependent dwell-time \cite{Briat:13b,Briat:14f,Zhang:15a}. Notably, minimum dwell-time stability conditions for linear impulsive systems have been obtained in \cite{Briat:11l} following the ideas developed in \cite{Geromel:06b} in the context of switched systems. It is shown in \cite{Briat:11l} that despite a sufficient stability condition can be expressed as a tractable linear matrix inequality problem, this type of conditions are only applicable to linear time-invariant systems without uncertainties and are impossible to convert into design conditions, even in the simple case of state-feedback design \cite{Briat:11l,Briat:12h,Briat:13d}. These drawbacks motivated the consideration of \emph{looped-functionals} \cite{Seuret:12,Briat:11l,Briat:12h,Briat:15f}, a particular class of indefinite functionals satisfying a certain boundary condition, referred to as the \emph{looping-condition}. These functionals have the merit of leading to conditions that are convex in the matrices of the system, hence easily applicable to linear uncertain systems with time-varying uncertainties and to nonlinear systems \cite{Peet:14}. The price to pay for these interesting properties is that the resulting conditions are infinite-dimensional semidefinite programs, that may then be solved using discretization techniques \cite{Gu:01b,Allerhand:11} or sum of squares programming \cite{Parrilo:00,Briat:13d}. Due to the presence of additional infinite-dimensional decision functions, the complexity of the conditions may not scale very well with the dimension of the system and/or the number of basis functions used to express infinite-dimensional variables.  Moreover, looped-functionals are limited to stability analysis and are difficult to consider for design purposes. To circumvent this problem, \emph{clock-dependent conditions} have been considered in \cite{Briat:13d,Briat:14a,Briat:14f,Briat:15f} where it is shown that these conditions possess the same advantages than looped-functional conditions (i.e. the possibility of considering nonlinear and uncertain linear systems) together with the additional possibility of using them for design purposes. Despite these frameworks have been recently shown to be theoretically equivalent in \cite{Briat:15f}, because the same accuracy is attained with a lower computational cost, it is then preferable to use clock-dependent stability conditions rather than looped-functional-based ones.

We propose to address here the case of stochastic impulsive systems where the continuous-time part consists of a linear stochastic differential equation with state multiplicative noise \cite{Oksendal:03} and the discrete-time part is a stochastic difference equation \cite{Rodkina:11}. The impulses arrival times are considered here as purely time-dependent and deterministic. Stochastic hybrid systems have been extensively studied in the literature; see e.g. \cite{Hou:01,Teel:14} and references therein. However, very few address the case where both parts of the impulsive system are affected by noise \cite{Chen:08,Feng:10} and, when this is the case, average dwell-time conditions, state-dependent impulse times or stochastically arriving impulses \cite{Hespanha:06,Antunes:13} are most of the time only considered.

The main goal of the paper is therefore to develop stability conditions in the same spirit  as in  \cite{Briat:11l,Briat:13d,Briat:14f,Briat:15f} where the deterministic case was considered. To this aim, we consider the notion of dwell-times and obtain a necessary and sufficient condition characterizing the mean-square stability under constant dwell-time, and sufficient conditions establishing the mean-square stability under ranged and minimum dwell-time conditions. These conditions, albeit stated in a more implicit way, naturally generalize those obtained in the deterministic setting as they reduce to the deterministic conditions when the noise-related terms are set to zero. Due to the implicit structure of the conditions, they cannot be checked per se. In order to overcome this difficulty, and in the same spirit as in \cite{Briat:13d,Briat:14f,Briat:15f}, lifted versions of the conditions expressed as clock-dependent linear matrix inequalities are considered. These conditions have the benefits of being convex in the matrices of the system, a property that enables their use in the contexts of uncertain systems and control design. Because these conditions are infinite-dimensional, they cannot be checked directly and need to be relaxed. Possible relaxation methods include the approximation of infinite-dimensional variables using a piecewise-linear approximation \cite{Gu:01b,Allerhand:11} or the use of sum of squares programming \cite{Parrilo:00,Briat:13d}. It is emphasized that these relaxed conditions are asymptotically exact in the sense that they can approximate arbitrarily well the original conditions they have been derived from. These conditions are shown to include those of \cite{Shaked:14}, which characterize the mean-square stability stochastic linear switched systems, by exploiting the possibility of representing switched systems as impulsive systems. The approach is then non-conservatively extended to state-feedback design, for which convex conditions are also obtained. Finally, the analysis and control of aperiodic stochastic sampled-data systems is performed using the proposed method by reformulating the considered sampled-data system into an impulsive system. It is worth mentioning here that sampled-data systems driven by multiplicative noise do not seem to have been thoroughly considered in the literature. This paper therefore fills this gap by providing  tractable conditions for both the analysis and the control of such systems. Various comparative examples demonstrate the accuracy and tractability of the approach.

\textbf{Outline:} The structure of the paper is as follows: in Section \ref{sec:preliminary} preliminary definitions and results are given. Section \ref{sec:stability} is devoted to dwell-time stability analysis while Section \ref{sec:stabz} addresses dwell-time stabilization. Sampled-data systems are finally treated in Section \ref{sec:SD}. Examples are considered in the related sections.

\textbf{Notations:} The cone of symmetric (positive definite) matrices of dimension $n$ is denoted by $\mathbb{S}^n$ ($\mathbb{S}^n_{\succ0}$). The sets of integers and whole numbers are denoted by $\mathbb{N}$ and $\mathbb{N}_0$, respectively. Given a vector $v$, its 2-norm is defined as $||v||_2=(v^Tv)^{1/2}$. For a square matrix $A$, we define $\He[A]:=A+A^T$. The symbols $\oplus$ and $\otimes$ are used for denoting the Kronecker sum and product, respectively.

\section{Preliminaries}\label{sec:preliminary}


From now on, the following class of linear stochastic impulsive system
\begin{equation}\label{eq:mainsyst}
\begin{array}{rcl}
  d x(t)&=&[Ax(t)+B_{c}^1u_c(t)]dt+E_cx(t) dW_1(t)\\
  &&+B_{c}^2u_c(t)dW_2(t),\ t\ne t_k\\
  x(t_k^+)&=&Jx(t_k)+B_{d}^1u_d(k)+E_{d}x(t_k)\nu_1(k)\\
  &&+B_{d}^2u_d(k)\nu_2(k),\ k\in\mathbb{N}\\
  x(0)&=&x_0
\end{array}
\end{equation}
is considered where $x,x_0\in\mathbb{R}^n$, $u_c\in\mathbb{R}^{m_c}$ and $u_d\in\mathbb{R}^{m_d}$ are the state of the system, the initial condition, the continuous control input and the discrete control input, respectively. The notation $x(t^+)$ is a shorthand for $\lim_{s\downarrow t}x(s)$, i.e. the trajectories are assumed to be left-continuous. The sequence of impulse instants $\{t_k\}_{k\in\mathbb{N}}$ is defined such that $T_k:=t_{k+1}-t_k\ge T_{min}$ for some $T_{min}>0$. This then implies that $\{t_k\}_{k\in\mathbb{N}_0}$, $t_0=0$, is increasing without bound. The processes $W_1(t)$ and $W_2(t)$ are independent scalar Wiener processes; i.e. for $i=1,2$, we have that $W_{i}(0)=0$, $W_{i}(t)$ is almost surely everywhere continuous and has independent increments $W_i(t)-W_i(s)$, $0\le s<t$, that are normally distributed with zero mean and variance $t-s$. The sequence $\{\nu_i(k)\}_{k\in\mathbb{N}_0}$, $i=1,2$, is sequence of independent identically distributed random variables with zero mean and unit variance that are independent of $x(t_k)$ for all $k\in\mathbb{N}$. Let $(\Omega,\mathcal{F},(\mathcal{F}_{t,k})_{t\ge0,k\in\mathbb{N}_0},\mathbb{P})$ be a complete probability space with $\sigma$-algebra $\mathcal{F}$ and natural (hybrid) filtration $\mathcal{F}_{t,k}$ (for more details on hybrid filtrations see \cite{Teel:14b}).  Let $\E[\cdot]$ be the expectation operator with respect to $\mathbb{P}$.

\begin{define}
  We say that the system \eqref{eq:mainsyst} with $u_d,u_c\equiv0$ is mean-square asymptotically stable if $\E[||x(t)||_2^2]\to0$ as $t\to\infty$.
\end{define}

\begin{lemma}\label{lem:stab}
  Let us consider a sequence $\{t_k\}$ for which $0<T_{min}\le T_k\le T_{max}<\infty$ for all $k\in\mathbb{N}_0$. Then, the system \eqref{eq:mainsyst} with $u_d,u_c\equiv0$ is mean-square asymptotically stable if and only if $\E[||x(t_k^+)||_2^2]\to0$ as $k\to\infty$.
\end{lemma}
\begin{proof}
It is clear that if the system is mean-square asymptotically stable then $\E[||x(t_k^+)||_2^2]\to0$ as ${k\to\infty}$. To prove the converse, first note that, when $u_d,u_c\equiv0$, $\frac{d\E[||x(t)||_2^2]}{dt}=\E[x(t)^T(\He[A]+E_c^TE_c)x(t)]\le\beta\E[||x(t)||_2^2]$ for all $t\in(t_k,t_{k+1})$ and for any large enough $\beta>0$. This then implies that $\E[||x(t_k+\tau)||_2^2]\le   e^{\beta T_{max}}\E[||x(t_k^+)||_2^2]$, which implies in turn that if $\E[||x(t_k^+)||_2^2]\to0$ as $k\to\infty$, then $\sup_{\tau\in(0,T_k]}\E[||x(t_k+\tau)||_2^2]\to0$ as $k\to\infty$. The proof is complete.
\end{proof}

The above result remains valid as long as the impulses are persistent, i.e. $T_k<\infty$. In the case where the impulses are not persistent (i.e. there exists a $k^*\in\mathbb{N}_0$ such that $T_{k}<\infty$ for $k=0,\ldots,k^*-1$ and $T_{k^*}=\infty$), then the asymptotic mean-square stability of \eqref{eq:mainsyst} eventually becomes equivalent to the asymptotic mean-square stability of the continuous-time part of \eqref{eq:mainsyst}. Note also that this result remains valid when state-dependent control inputs are considered (e.g. state-, output- and dynamic-output feedback).
\begin{remark}\label{rem:exp}
  By considering the change of variables $z(t)=e^{\alpha t}x(t)$ and applying the above result to $z(t)$, we immediately get that $\E[||x(t)||_2^2]$ exponentially converges to 0 with rate $2\alpha$ as $t\to\infty$ if and only if  $\E[||x(t_k^+)||_2^2]$ geometrically converges to 0 with rate $e^{-2\alpha T_{min}}$ as $k\to\infty$.
\end{remark}


\section{Mean-square stability of stochastic linear impulsive systems}\label{sec:stability}

This section is devoted to the mean-square stability analysis of the system \eqref{eq:mainsyst} with no input (i.e. $u_c\equiv0$ and $u_d\equiv0$) under constant, ranged and minimum dwell-time. Before stating the main results of this section, it is convenient to introduce here the following definition:
\begin{define}
  The fundamental solution of the continuous-time part of the system \eqref{eq:mainsyst} with $u_c\equiv0$ and $u_d\equiv0$ is denoted by $\Phi:\mathbb{R}_{\ge0}\to\mathbb{R}^{n\times n}$ where
  \begin{equation}\label{eq:dPhi}
    d\Phi(t)=A\Phi(t)dt+E_c\Phi(t)dW_1(t),\ t\ge0
  \end{equation}
  where $\Phi(0)=I$ and $W_1(t)$ is defined as in \eqref{eq:mainsyst}.
\end{define}
Note that for every $t\ge0$, $\Phi(t)$ is an $\mathcal{F}_{t,k}$-measurable random variable. This fundamental solution therefore naturally generalizes the deterministic one, which is readily retrieved by setting $E_c=0$. Associated with $\Phi(t)$, we define for any $Z\in\mathbb{S}^{n}$ the following quadratic expression
\begin{equation}\label{eq:xi}
  \Xi_Z(s):=\Phi(s)^TZ\Phi(s),\ s\ge0.
\end{equation}
With these definitions in mind, we can now move forward to the main results of the section.

\subsection{Constant dwell-time}

Let us first consider the case of constant dwell-time, that is, the case where, for some $\bar{T}>0$, we have that $T_k=\bar{T}$ for all $k\in\mathbb{N}_0$. We then have the following result:
\begin{theorem}\label{th:periodic}
  The following statements are equivalent:
  \begin{enumerate}[(i)]
    \item\label{item:periodic:1} The system \eqref{eq:mainsyst} with $u_c\equiv0,u_d\equiv0$  is mean-square asymptotically stable under constant dwell-time, that is, for the sequence of impulse times verifying $T_k=\bar{T}$, $k\in\mathbb{N}$,
\item\label{item:periodic:1b} The matrix $\mathcal{M}(\bar T):=\exp(\mathcal{A}\bar{T})\mathcal{J}$ defined with
\begin{equation*}
      \hspace{-10pt}\mathcal{A}:=A\oplus A+E_c\otimes E_c\ \textnormal{and}\ \mathcal{J}:=J\otimes J+E_d\otimes E_d
\end{equation*}
is Schur stable.
    \item There exists a matrix $P\in\mathbb{S}^n_{\succ0}$ such that
    \begin{equation}\label{eq:expon}
      \E[J^T\Xi_P(\bar{T}) J+E_d^T\Xi_P(\bar{T}) E_d]-P\prec0
    \end{equation}
    holds where $\Xi_P(\cdot)$ is defined in \eqref{eq:xi}.\label{item:periodic:2}
    \item\label{item:periodic:3} There exist a matrix-valued function $S:[0,\bar{T}]\to\mathbb{S}^n$, $S(0)\succ0$, and a scalar $\eps>0$ such that the conditions
    \begin{equation}\label{eq:p31}
        -\dot{S}(\tau)+A^TS(\tau)+S(\tau)A+E_c^TS(\tau)E_c\preceq0
    \end{equation}
    and
    \begin{equation}\label{eq:p32}
          J^TS(\bar{T})J-S(0)+E_d^TS(\bar{T})E_d+\eps I\preceq0
    \end{equation}
    hold for all $\tau\in[0,\bar{T}]$.
    \item\label{item:periodic:4} There exist a matrix-valued function $S:[0,\bar{T}]\to\mathbb{S}^n$, $S(\bar{T})\succ0$, and a scalar $\eps>0$ such that the conditions
    \begin{equation}\label{eq:p41}
        \dot{S}(\tau)+A^TS(\tau)+S(\tau)A+E_c^TS(\tau)E_c\preceq0
    \end{equation}
    and
    \begin{equation}\label{eq:p42}
          J^TS(0)J-S(\bar{T})+E_d^TS(0)E_d+\eps I\preceq0
    \end{equation}
        hold for all $\tau\in[0,\bar{T}]$.
  \end{enumerate}
\end{theorem}
\begin{proof}
\textbf{Proof that \eqref{item:periodic:2} $\Rightarrow$ \eqref{item:periodic:1}.} The proof of this implication follows from the fact that by pre- and post-multiplying \eqref{eq:expon} by $x(t_k^+)^T$ and $x(t_k^+)$, respectively, we immediately get that $\E[||x(t_{k+1}^+)||_{2,P}^2]-(1-\eps)\E[||x(t_k^+)||_{2,P}^2]\le0$ for some $\eps\in(0,1)$, all $k\in\mathbb{N}_0$ and where $||v||_{2,P}^2:=v^TPv$. This implies that $\E[||x(t_k^+)||_{2,P}^2]\to0$ as $k\to\infty$, proving then the mean-square asymptotic (exponential) stability of the system (invoking Lemma \ref{lem:stab} and Remark \ref{rem:exp}).

\textbf{Proof that \eqref{item:periodic:1} $\Rightarrow$ \eqref{item:periodic:2}.} The proof of this statement is based on the explicit construction of a matrix $P\in\mathbb{S}^n_{\succ0}$ that verifies the condition \eqref{eq:expon} whenever the system \eqref{eq:mainsyst}  is mean-square asymptotically (exponentially) stable. To do this, let us define first the following expression
\begin{equation}\label{eq:Qsyst}
  \begin{array}{rcl}
    \dot{Q}(t)&=&A^TQ(t)+Q(t)A+E_c^TQ(t)E_c,\ t\ne t_k\\
    Q(t_k^+)&=&J^TQ(t_k)J+E_d^TQ(t_k)E_d,\ k\in\mathbb{N}
  \end{array}
\end{equation}
with $Q_0=Q(0)=Y$ for any $Y\in\mathbb{S}^n_{\succ0}$. Note that the mean-square exponential stability of the system is equivalent to the exponential stability of the above matrix-valued linear differential equations since $\E[x(t)^TYx(t)]=x_0^TQ(t)x_0$. Let us
define  $P^*=\sum_{k=0}^\infty Q_k\succ0$, $Q_k:=Q(t_k^+)\succeq0$, $Q_0=Y\succ0$. Note that the latter sum is well-defined because \eqref{eq:Qsyst} is exponentially stable and also observe that
\begin{equation}\label{eq:kdslkdslk}
  \E\left[J^T\Xi_{Q_k}(\bar T)J+E_d^T\Xi_{Q_k}(\bar  T)E_d\right]=Q_{k+1}.
  \end{equation}
 Substituting then $P^*$ in place of $P$ in \eqref{eq:expon} and using \eqref{eq:kdslkdslk} yield
\begin{equation}
\sum_{k=1}^\infty Q_k-\sum_{k=0}^\infty Q_k=-Q_0=-Y\prec0,
\end{equation}
which proves the result.

\textbf{Proof that \eqref{item:periodic:1} $\Leftrightarrow$ \eqref{item:periodic:1b}.} The mean-square asymptotic stability of \eqref{eq:mainsyst} is equivalent to the asymptotic stability of \eqref{eq:Qsyst}. Vectorizing this matrix differential equation yields the system $\dot{z}(t)=\mathcal{A}z(t)$, $t\ne t_k$, and $z(t_k^+)=\mathcal{J}z(t_k)$, $k\in\mathbb{N}$. Since the impulses arrive periodically, then the stability of this system is equivalent to the stability of the discretized system $z(t_{k+1}^+)=\exp(\mathcal{A}\bar T)\mathcal{J}z(t_k^+)$ and the result follows.

\textbf{Proof that \eqref{item:periodic:3} $\Rightarrow$ \eqref{item:periodic:2}.} Integrating \eqref{eq:p31} with $S(0)=P$, for simplicity, and using the definition of $\Phi$, we get that $S(\tau)\succeq \E[\Phi(\tau)^TS(0)\Phi(\tau)]$. Substituting then this expression in \eqref{eq:p32} yields the condition \eqref{eq:expon}. The proof is complete.

\textbf{Proof that \eqref{item:periodic:2} $\Rightarrow$ \eqref{item:periodic:3}.} Assume that \eqref{eq:expon} holds and define $S^*(\tau)=\E[\Phi(\tau)^TS(0)\Phi(\tau)]$, $S(0)=P$. This gives that $-\dot{S}^*(\tau)+A^TS^*(\tau)+S(\tau)A^*+E_c^TS(\tau)^*E_c=0$, hence \eqref{eq:p31} holds. Substituting then the value $S^*(\bar{T})$ in the place of $S(\bar{T})$ in \eqref{eq:p32} yields an expression that is identical to \eqref{eq:expon}. The condition \eqref{eq:p32} then readily follows.

\textbf{Proof that \eqref{item:periodic:3} $\Leftrightarrow$ \eqref{item:periodic:4}.} The equivalence follows from the change of variables $S(\tau)\leftarrow S(\bar{T}-\tau)$.
\end{proof}

The conditions stated in the two last statements are referred to as \emph{clock-dependent conditions} as they explicitly depend on the time $\tau$ (the clock value) elapsed since the last impulse. At each event, the clock is reset to 0 (i.e. $\tau(t_k^+)=0$) and then grows continuously with slope 1 until the next impulse time. Clocks are extensively considered in the analysis of hybrid systems and timed automata; see e.g. \cite{Baier:08,Goebel:09}. Clocks here are used to measure the current dwell-time value and explicitly consider it in the conditions.

It is interesting to note that when the system without any control input is deterministic, i.e. $E_c=0$ and $E_d=0$, then the conditions stated in the previous result reduce to those of \cite{Briat:13d}, emphasizing then their higher degree of generality. However, unlike in the deterministic setting, the condition \eqref{eq:expon} is not directly tractable in this form due to the presence of the expectation operator and random matrices that are difficult to compute. The conditions of the other statements, although stated as infinite-dimensional LMI problems, can be turned into finite-dimensional tractable conditions using several relaxation techniques. This will be discussed in more details in Section \ref{sec:comp}.

\subsection{Ranged dwell-time}

Let us consider now the ranged dwell-time case, that is, the case where $T_k\in[T_{min},T_{max}]$, $0<T_{min}\le T_{max}<\infty$, for all $k\in\mathbb{N}_0$. We then have the following result:
\begin{theorem}\label{th:aperiodic}
  The following statements are equivalent:
  \begin{enumerate}[(i)]
\item \label{item:aperiodic:1} There exists a matrix $P\in\mathbb{S}^n_{\succ0}$ such that
    \begin{equation}\label{eq:expon2}
      \E[J^T\Xi_P(\theta) J+E_d^T\Xi_P(\theta) E_d]-P\prec0
    \end{equation}
    holds for all $\theta\in[T_{min},T_{max}]$ where $\Xi_P(\cdot)$ is defined in \eqref{eq:xi}.
    \item  \label{item:aperiodic:2}There exist a matrix-valued function  $S:[0,T_{max}]\to\mathbb{S}^n$, $S(0)\succ0$, and a scalar $\eps>0$ such that the conditions
    \begin{equation}\label{eq:kdskddmlkdsqdmlkswmlk}
        -\dot{S}(\tau)+A^TS(\tau)+S(\tau)A+E_{c}^TS(\tau)E_{c}\preceq0
    \end{equation}
and
\begin{equation}
  J^TS(\theta)J-S(0)+E_d^TS(\theta)E_d+\eps I\preceq0
\end{equation}
hold for all $\tau\in[0,T_{max}]$ and all $\theta\in[T_{min},T_{max}]$.
  \end{enumerate}
  Moreover, when one of the above equivalent statements holds, then the system \eqref{eq:mainsyst} with $u_c\equiv0$ and $u_d\equiv0$ is mean-square asymptotically stable under ranged dwell-time $(T_{min},T_{max})$, that is, for any sequence of impulse times verifying $T_k\in[T_{min},T_{max}]$ for all $k\in\mathbb{N}$.
\end{theorem}
\begin{proof}
The proof of this result follows from the same lines as the proof of Theorem \ref{th:periodic}.
\end{proof}

\subsection{Minimum dwell-time}

We finally consider the minimum dwell-time case, that is, $T_k\ge \bar{T}$ for all $k\in\mathbb{N}_0$. We then have the following result:
\begin{theorem}\label{th:minDT}
  The following statements are equivalent:
  \begin{enumerate}[(i)]
    \item\label{item:minDT:1} There exists a $P\in\mathbb{S}^n_{\succ0}$ such that the LMIs
    \begin{equation}
        \E[J^T\Xi_P(\theta) J+E_d^T\Xi_P(\theta) E_d]-P\prec0
    \end{equation}
    and
    \begin{equation}
      A^TP+PA+E_c^TPE_c\prec0
    \end{equation}
    hold for all $\theta\ge\bar{T}$ where $\Xi_P(\cdot)$ is defined in \eqref{eq:xi}.
   \item\label{item:minDT:2} There exists a $P\in\mathbb{S}^n_{\succ0}$ such that the LMIs
    \begin{equation}\label{eqmldsm:kdslkdkl}
        \E[J^T\Xi_P(\bar{T}) J+E_d^T\Xi_P(\bar{T}) E_d]-P\prec0
    \end{equation}
    and
    \begin{equation}\label{eq:dskdksqldkkmlkml}
      A^TP+PA+E_c^TPE_c\prec0
    \end{equation}
    hold  where $\Xi_P(\cdot)$ is defined in \eqref{eq:xi}.
    \item\label{item:minDT:3} There exist a matrix-valued function  $S:\mathbb{R}_{\ge0}\to\mathbb{S}^n$, $S(\bar{T})\succ0$, $S(\bar{T}+s)=S(\bar{T})$, $s\ge0$, and a scalar $\eps>0$ such that the conditions
    \begin{equation}\label{eq:dksldk670}
  A^TS(\bar{T})+S(\bar{T})A+E_c^TS(\bar{T})E_c\prec0,
\end{equation}
    \begin{equation}\label{eq:dksldk67}
        \dot{S}(\tau)+A^TS(\tau)+S(\tau)A+E_{c}^TS(\tau)E_{c}\preceq0
    \end{equation}
and
\begin{equation}\label{eq:dksldk672}
  J^TS(0)J-S(\bar{T})+E_d^TS(0)E_d+\eps I\preceq0
\end{equation}
hold for all $\tau\in[0,\bar{T}]$.
  \end{enumerate}
      Moreover, when one of the above equivalent statements holds, then the system \eqref{eq:mainsyst} with $u_c\equiv0$ and $u_d\equiv0$ is mean-square asymptotically stable under minimum dwell-time $\bar{T}$, that is, for any sequence of impulse times verifying $T_k\ge\bar{T}$ for all $k\in\mathbb{N}$.
\end{theorem}
\begin{proof}
  The proof that \eqref{item:minDT:2} is equivalent to \eqref{item:minDT:3} follows from Theorem \ref{th:periodic}. The proof that \eqref{item:minDT:1} implies \eqref{item:minDT:2} is also immediate. Let us then focus on the reverse implication. Define first the function $f_{\bar{T}}(\theta):=\E[\Phi(\bar{T}+\theta)^TP\Phi(\bar{T}+\theta)]$. Then, we have that
  \begin{equation}
    \dfrac{df_{\bar{T}}}{d\theta}=\E[\Phi(\bar{T}+\theta)^T[\He[PA]+E_c^TPE_c]\Phi(\bar{T}+\theta)].
  \end{equation}
  Using now \eqref{eq:dskdksqldkkmlkml}, we can conclude that $\frac{df_{\bar{T}}}{d\theta}\preceq0$ for all $\theta\ge0$. This, in turn, implies that
  \begin{equation}
  \begin{array}{l}
         \hspace{-5mm}\E[J^T\Xi_P(\bar{T}+\theta) J+E_d^T\Xi_P(\bar{T}+\theta) E_d]\\
         \hspace{15mm}\preceq\E[J^T\Xi_P(\bar{T}) J+E_d^T\Xi_P(\bar{T}) E_d]
  \end{array}
  \end{equation}
  for any sufficiently small $\theta\ge0$. Noting finally that the above inequality also holds when substituting $\bar{T}$ by $\bar{T}+\mu$ for any $\mu\ge0$ we get that
   \begin{equation}
   \begin{array}{l}
          \hspace{-5mm}\E[J^T\Xi_P(\bar{T}+\theta') J+E_d^T\Xi_P(\bar{T}+\theta') E_d]\\
          \hspace{15mm}\preceq\E[J^T\Xi_P(\bar{T}) J+E_d^T\Xi_P(\bar{T}) E_d]
   \end{array}
  \end{equation}
  for all $\theta':=\theta+\mu\ge0$. Using finally \eqref{eqmldsm:kdslkdkl} yields the result. The proof is complete.
\end{proof}

\subsection{Computational aspects}\label{sec:comp}

The conditions of the theorems stated in the previous sections are infinite-dimensional LMI feasibility problems which cannot be verified directly. In what follows, we describe two relaxation methods turning the original untractable conditions into tractable ones.  Note that even though we only provide these relaxations for Theorem \ref{th:periodic}, \eqref{item:periodic:3}, similar ones can be obtained for the conditions of Theorem \ref{th:aperiodic} and Theorem \ref{th:minDT}.

\subsubsection{Piecewise linear approach}

The first method, referred to as the piecewise linear approximation, proposes to impose a piecewise linear structure to the general matrix-valued functions involved in the conditions; see e.g. \cite{Allerhand:11}. The following result states the conditions that approximate those of Theorem \ref{th:periodic}, \eqref{item:periodic:3}:
\begin{proposition}\label{prop:discretization}
  Let $N\in\mathbb{N}$. The following statements are equivalent:
  \begin{enumerate}
    \item The conditions of Theorem \ref{th:periodic}, \eqref{item:periodic:3} hold with the piecewise-linear matrix-valued function $S(\tau)$ given by
        \begin{equation}\label{eq:pieceS}
          S(\tau)=\dfrac{S_{i+1}-S_i}{\bar{T}/N}\left(\tau-\dfrac{i\bar{T}}{N}\right)+S_i
        \end{equation}
        for $\tau\in[i\bar{T}/N,(i+1)\bar{T}/N]$, $S_i\in\mathbb{S}^n$, $i=0,\ldots,N-1$.
    \item There exist matrices $S_i\in\mathbb{S}^n$, $i=1,\ldots,N$, $S_0\succ0$, and a scalar $\eps>0$ such that the LMIs
  \begin{equation}\label{eq:pieceS1}
        -\dfrac{S_{i+1}-S_i}{\bar{T}/N}+A^TS_i+S_iA+E_c^TS_iE_c\preceq0
    \end{equation}
    \begin{equation}\label{eq:pieceS2}
           -\dfrac{S_{i+1}-S_i}{\bar{T}/N}+A^TS_{i+1}+S_{i+1}A+E_c^TS_{i+1}E_c\preceq0
    \end{equation}
and
\begin{equation}\label{eq:pieceS3}
  J^TS_NJ-S_0+E_d^TS_NE_d+\eps I\preceq0
\end{equation}
hold for all $i=0,\ldots,N-1$.
  \end{enumerate}
  When one of the above statements holds, then the conditions of Theorem \ref{th:periodic}, \eqref{item:periodic:3} hold with the computed piecewise-linear matrix $S(\tau)$.
\end{proposition}
\begin{proof}
  The proof follows from a convexity argument. By substituting the piecewise-linear expression of $S(\tau)$ given by \eqref{eq:pieceS} into the conditions \eqref{eq:p32} and \eqref{eq:p31} gives \eqref{eq:pieceS3} and
  \begin{equation}
      -\dfrac{S_{i+1}-S_i}{\bar{T}/N}+\He[S(\tau)A]+E_c^TS(\tau)E_c\preceq0,
  \end{equation}
  respectively. Noting then that the above LMI is affine in $\tau$, then for each $i=0,\ldots,N-1$, it is necessary and sufficient to check the above LMI at the vertices of the interval $[i\bar{T}/N,(i+1)\bar{T}/N]$, that is, at the values $i\bar{T}/N$ and $(i+1)\bar{T}/N$. The equivalence follows from the losslessness of the manipulations.
\end{proof}

As the conditions stated in the above result are finite-dimensional, they can be solved using standard SDP solvers such as SeDuMi \cite{Sturm:01a}. Note, moreover, than using similar arguments as in \cite{Xiang:15a}, we can prove that if the conditions of Theorem \ref{th:periodic}, \eqref{item:periodic:3} are feasible, then there exists an integer $N^*$ such that the conditions of Proposition \ref{prop:discretization}, (b) are feasible for all $N\ge N^*$ .

\subsubsection{Sum of squares programs}

Another possible relaxation relies on the use of sum of squares programming \cite{Parrilo:00,sostools3} where we impose a polynomial structure with fixed degree to the matrix-valued functions. The following result states conditions that approximate those of Theorem \ref{th:periodic}, \eqref{item:periodic:3} and that can be easily checked using SOSTOOLS \cite{sostools3} and the semidefinite programming solver SeDuMi \cite{Sturm:01a}:
\begin{proposition}\label{prop:sos1}
    Let $\eps,\nu,\bar{T}>0$ be given and  assume that the following sum of squares program
    \begin{equation*}
  \begin{array}{l}
    \textnormal{Find polynomial matrices }S,\Gamma:\mathbb{R}\to\mathbb{S}^n\textnormal{\ such that}\\
     S(0)-\nu I_n\succeq0\\
     \Gamma(\tau)\ \textnormal{is SOS}\\
     \dot{S}(\tau)-\He[S(\tau)A]-E_c^TS(\tau)E_c-\Gamma(\tau)\tau(\bar{T}-\tau)\textnormal{ is SOS}\\
    S(0)-J^TS(\bar{T})J-E_d^TS(\bar{T})E_d-\eps I\succeq0
  \end{array}
\end{equation*}
is feasible. Then the conditions of Theorem \ref{th:periodic}, \eqref{item:periodic:3} hold with the computed polynomial matrix $S(\tau)$ and the system \eqref{eq:mainsyst} is mean-square asymptotically stable under constant dwell-time $\bar{T}$.
\end{proposition}


Regarding the conservatism, it can be shown using the same arguments as in \cite{Briat:14f} that if the conditions of Theorem \ref{th:periodic}, \eqref{item:periodic:3} are feasible, then there exists a sufficiently large integer $d$ such that the above SOS program is feasible for some polynomials $S,\Gamma$ of degree at least $2d$. Finally, it is important to stress that SOS conditions are in general more tractable than those obtained using the piecewise linear approximation that often require a large discretization order; see \cite{Briat:14f,Briat:15f}.

\subsection{Application to switched systems}

Interestingly, the results developed in this section also applies to linear stochastic switched systems. To emphasize this, let us consider the stochastic switched system
\begin{equation}\label{eq:switched}
  \begin{array}{lcl}
    dy(t)&=&G_{\sigma(t)}y(t)dt+H_{\sigma(t)}y(t)dW_1(t),\ y(0)=y_0
  \end{array}
\end{equation}
where $y,y_0\in\mathbb{R}^n$ and $W_1(t)$ are the state of the system, the initial condition and the standard Wiener process, respectively. The switching signal $\sigma:\mathbb{R}_{\ge0}\to\{1,\ldots,N\}$, for some finite $N\in\mathbb{N}$, is piecewise constant and describes the evolution of the mode of the switched system. This system can be reformulated in the form \eqref{eq:mainsyst} with the matrices
\begin{equation}\label{eq:swicthed matrices}
  A=\diag_{k=1}^N[G_k],\ E_c=\diag_{k=1}^N(H_k)\ \textnormal{and}\ J_{ij}=(e_ie_j^T)\otimes I_n,
\end{equation}
for all $i,j=1,\ldots,\mu$, $i\ne j$, where $\{e_i\}_{i=1,\ldots,N}$ is the standard basis of $\mathbb{R}^N$. Note, however, that we have here multiple jump maps $J_{ij}$ (actually reset maps here). It is immediate to see that the derived stability conditions can be straightforwardly extended to address this case, and we get the following result adapted from Theorem \ref{th:minDT}:

\begin{corollary}\label{th:minDTsw}
Assume that there exists a block-diagonal matrix-valued function $R=\diag_{i=1}^N(R_i)$, $R_i:[0,\bar{T}]\to\mathbb{S}^n$, $R(\bar{T})\succ0$, such that the conditions
  \begin{equation}
   G_i^TR_i(\bar{T})+R_i(\bar{T})G_i+H_i^TR_i(\bar{T})H_i\prec0,
\end{equation}
 \begin{equation}
        \dot{R}_i(\tau)+G_i^TR_i(\tau)+R_i(\tau)G_i+H_i^TR_i(\tau)H_i\preceq0
    \end{equation}
    and
    \begin{equation}
  R_i(0)-R_j(\bar{T})+\eps I\preceq0
\end{equation}
hold  for all $\tau\in[0,\bar{T}]$ and for all $i,j=1,\ldots,N$, $i\ne j$.

Then, the linear stochastic switched system \eqref{eq:switched} is mean-square asymptotically stable under minimum dwell-time $\bar{T}$.
\end{corollary}
\begin{proof}
  The proof is based on the reformulation \eqref{eq:mainsyst}-\eqref{eq:swicthed matrices}. Noting then the system is block-diagonal, then it is enough to choose a matrix-valued function $R(\tau)$ that is also block-diagonal. Substituting the model in the conditions of Theorem \ref{th:minDT} and expanding them yield the result.
\end{proof}

 It is interesting to note that the above condition, although formulated in a more compact form, include those obtained in \cite{Shaked:14}. If we indeed apply now the piecewise linear approximation, then we get exactly the conditions of \cite[Theorem 1]{Shaked:14}.

\subsection{Examples}

We now apply some of the previously developed results to some academic examples.
\begin{example}
  Let us consider the system \eqref{eq:mainsyst} with the matrices \cite{Briat:11l,Briat:12h,Briat:13d}
  \begin{equation}\label{eq:ex1}
    A=\begin{bmatrix}
      -1 & 0\\
      1 & -2
    \end{bmatrix},\ J=\begin{bmatrix}
      2 & 1\\
      1 & 3
    \end{bmatrix},\ E_c=\kappa I_2, E_d=\delta I_2,
  \end{equation}
  $B_c^1=B_c^2=B_d^1=B_d^2=0$ (i.e. no control input) for some scalars $\kappa,\delta\ge0$. We then choose several values for the parameters $\delta$ and $\kappa$. The case $(\kappa,\delta)=(0,0)$ corresponds to the deterministic case of \cite{Briat:13d}. For each of these values for the parameters, we solve the sum of squares program of Proposition \ref{prop:sos1} for the constant dwell-time and the minimum dwell-time cases with matrix polynomials of degree 6. Note that to adapt Proposition \ref{prop:sos1} to the minimum dwell-time case, we simply have to add the constraint $\He[A^TS(\bar{T})]+E_c^TS(\bar{T})E_c\prec0$ to the program.

  The numerical results are gathered in Table \ref{tab:constant1} and Table \ref{tab:minDT1} where we can see that the deterministic results are indeed retrieved and that, as expected, stability deteriorates as we increase the value of the  parameters $\kappa$ and $\delta$. It is also interesting to observe that, for this example, the results for the constant and the minimum dwell-times are quite close. Note, however, that this is far from being a general rule.
In order to estimate the conservatism of the method, we consider the criterion of statement \eqref{item:periodic:1b} of Theorem \ref{th:periodic} and we find the results summarized in Table \ref{tab:constant1th} where we can see that the proposed method is very accurate for this system. This then implies that the estimates of the minimum dwell-time are also equally accurate.
\end{example}

\begin{table}[\tablepos]
\centering
\caption{Estimated smallest constant dwell-time $\bar{T}$ for the system \eqref{eq:mainsyst}-\eqref{eq:ex1} for various values for $\kappa$ and $\delta$ using Proposition \ref{prop:sos1} and matrix polynomials $S(\tau)$ and $\Gamma(\tau)$ of degree 6.}\label{tab:constant1}
  \begin{tabular}{|c||c|c|c|c|c|}
    \hline
    $\kappa$/$\delta$ & 0 & 0.6 & 1.2 & 1.8 & 2.4\\
    \hline
    \hline
    0 &   1.1406   & 1.1568 &   1.2031  &  1.2734 &   1.3595\\
    0.3 &     1.1918 &   1.2089&    1.2578 &   1.3319 &   1.4225\\
    0.6 & 1.3787 &   1.3992  &  1.4577   & 1.5458  &  1.6531\\
    0.9 & 1.8774  &  1.9073  &  1.9920   & 2.1184  &  2.2702\\
    1.2 & 3.9306  &  4.0011  &  4.1938 &   4.4765  &  4.8305\\
    \hline
  \end{tabular}
\end{table}

\begin{table}[\tablepos]
\centering
\caption{Smallest constant dwell-time $\bar{T}$ for the system \eqref{eq:mainsyst}-\eqref{eq:ex1} for various values for $\kappa$ and $\delta$ using Theorem \ref{th:periodic}, \eqref{item:periodic:1b}.}\label{tab:constant1th}
  \begin{tabular}{|c||c|c|c|c|c|}
    \hline
    $\kappa$/$\delta$ & 0 & 0.6 & 1.2 & 1.8 & 2.4\\
    \hline
    \hline
    0 & 1.1406  &  1.1568    &1.2030    &1.2732    &1.3593\\
    0.3 &  1.1918    &1.2089    &1.2577    &1.3317    &1.4223\\
    0.6 & 1.3787&    1.3992  &  1.4576  &  1.5456 &   1.6528\\
    0.9 & 1.8773   & 1.9072  &  1.9918  &  2.1181 &   2.2700\\
    1.2 & 3.9315  &  4.0005  &  4.1932  &  4.4752  &  4.8083\\
    \hline
  \end{tabular}
\end{table}

\begin{table}[\tablepos]
\centering
\caption{Estimated minimum dwell-time $\bar{T}$ for the system \eqref{eq:mainsyst}-\eqref{eq:ex1} for various values for $\kappa$ and $\delta$ using Proposition \ref{prop:sos1} and matrix polynomials $S(\tau)$ and $\Gamma(\tau)$ of degree 6.}\label{tab:minDT1}
  \begin{tabular}{|c||c|c|c|c|c|}
    \hline
    $\kappa$/$\delta$ & 0 & 0.6 & 1.2 & 1.8 & 2.4\\
    \hline
    \hline
    0 & 1.1406&    1.1568    &1.2031   & 1.2734  &  1.3595\\
    0.3 & 1.1918&    1.2089    &1.2578  &  1.3319  &  1.4225\\
    0.6 &  1.3787&    1.3992  &  1.4577 &   1.5458   & 1.6531\\
    0.9 & 1.8774 &   1.9073   & 1.9920  &  2.1184   & 2.2703\\
    1.2 & 3.9307  &  4.0012   & 4.1941&    4.4776  &  4.8565\\
    \hline
  \end{tabular}
\end{table}

\begin{example}
  We now consider the system \eqref{eq:mainsyst} with the matrices \cite{Briat:11l,Briat:12h,Briat:13d}
    \begin{equation}\label{eq:ex2}
    A=\begin{bmatrix}
      1 & 3\\
      -1 & 2
    \end{bmatrix},\ J=\begin{bmatrix}
      0.5 & 0\\
      0 & 0.5
    \end{bmatrix},\ E_c=\kappa I_2, E_d=\delta I_2,
  \end{equation}
  $B_c^1=B_c^2=B_d^1=B_d^2=0$ (i.e. no control input) for some scalars $\kappa,\delta\ge0$. For $\delta=\kappa=0$, this system is known  (see \cite{Briat:11l})  to be stable with maximum dwell-time $T_{max}\approx 0.4620$, i.e. for all $T_k\le0.4620$. Using the ranged dwell-time result with $T_{min}=0.01$, we get the results the maximal value for $T_{max}$ using a variation of the SOS program in Proposition \ref{prop:sos1} and a bisection approach. In order to evaluate the conservatism, we compare these results with the maximum value for $\theta$ for which the LMIs $P\succ0$ and $\mathcal{M}(\theta)^TP\mathcal{M}(\theta)-P\prec0$ are feasible for all $\theta\in[0.01,T_{max}]$ (quadratic stability condition). Since this problem is not directly solvable (note that sum of squares methods do not apply because of the presence of exponential terms), the interval $[0.01,T_{max}]$ is gridded with 201 points and the LMIs are checked on these points only. This leads to the results of Table \ref{tab:ranged_grid} where we can observe that slightly larger values for $T_{max}$ are found. However, this is at the price of a much higher computational cost (see \cite{Briat:14f,Briat:15f}) and gridding imprecision. Note also that the exponential conditions are only valid in the time-invariant case, while the conditions of Theorem \ref{th:aperiodic} are more flexible and apply also to systems affected by time-varying uncertainties/parameters.
\end{example}

\begin{table}[\tablepos]
\centering
\caption{Estimated ranged dwell-time $T_{max}$ for $T_{min}=0.01$ for the system \eqref{eq:mainsyst}-\eqref{eq:ex2} for various values for $\kappa$ and $\delta$ using an adaptation of Proposition \ref{prop:sos1} and matrix polynomials of degree 6.}\label{tab:ranged}
  \begin{tabular}{|c||c|c|c|c|c|}
    \hline
     $\kappa$/$\delta$ & 0 & 0.2 & 0.4 & 0.6 & 0.8\\
    \hline
    \hline
    0 &0.4620  &  0.4126    &0.2971    &0.1647    &0.0388\\
    0.75 & 0.3891    &0.3474    &0.2502    &0.1387    &0.0327\\
    1.5 &  0.2640   & 0.2357    &0.1698    &0.0941    &0.0221\\
    2.75 &0.1312  &  0.1171  &  0.0844 &   0.0467   & 0.0110\\
    3 &0.1154 &   0.1031&    0.0742 &   0.0411  &  0.0064\\
    \hline
  \end{tabular}
\end{table}

\begin{table}[\tablepos]
\centering
\caption{Estimated ranged dwell-time $T_{max}$ for $T_{min}=0.01$ for the system \eqref{eq:mainsyst}-\eqref{eq:ex2} for various values for $\kappa$ and $\delta$ using the gridded quadratic stability condition.}\label{tab:ranged_grid}
  \begin{tabular}{|c||c|c|c|c|c|}
    \hline
    $\kappa$/$\delta$ & 0 & 0.2 & 0.4 & 0.6 & 0.8\\
    \hline
    \hline
    0 & 0.4620 &   0.4126   & 0.2971  &  0.1647 &   0.0388\\
    0.75 & 0.3891 &   0.3474   & 0.2502    &0.1387   & 0.0327\\
    1.5 & 0.2640  &  0.2357  &  0.1698    &0.0941   & 0.0221\\
    2.75 & 0.1312 &   0.1171&    0.0844  &  0.0467&    0.0110\\
    3 & 0.1155  &  0.1031    &0.0742   & 0.0411 &   0.0411\\
    \hline
  \end{tabular}
\end{table}

\section{Mean-square stabilization of stochastic linear impulsive systems}\label{sec:stabz}

We extend here the results obtained in the previous section to address the stabilization problem by state-feedback. We first consider stabilization under ranged dwell-time and then derive stabilization conditions under minimum dwell-time.

\subsection{Stabilization under ranged dwell-time}

We consider in this section the following class of state-feedback control law
\begin{equation}\label{eq:SFranged}
  \begin{array}{rcl}
    u_c(t_k+\tau)&=&K_c(\tau)x(t_k+\tau),\ \tau\in(0,T_k]\\
    u_d(k)&=&K_dx(t_k)
  \end{array}
\end{equation}
where $T_k\in[T_{min},T_{max}]$, $k\in\mathbb{N}_0$. The continuous matrix-valued function $K_c:[0,T_{max}]\to\mathbb{R}^{m_c\times n}$ and the matrix $K_d\in\mathbb{R}^{m_d\times n}$ involved above are the gains of the controller to be determined. We then have the following result:
\begin{theorem}\label{th:aperiodicz}
The following statements are equivalent:
\begin{enumerate}
  \item There exists a matrix $P\in\mathbb{S}^n_{\succ0}$ such that the LMI
  \begin{equation}
  \begin{array}{l}
     \hspace{-7mm} \E[(J+B_d^1K_d)^T\Psi_1(\theta)(J+B_d^1K_d)]-P\\
        \hfill+E_d^T\Psi_1(\theta)E_d+K_d^T(B_d^2)^T\Psi_1(\theta)B_d^2K_d\prec0
  \end{array}
  \end{equation}
  holds for all $\theta\in[T_{min},T_{max}]$ where $\Psi_1(\theta):=\Phi_1(\theta)^TP\Phi_1(\theta)$ and
  \begin{equation}
  \begin{array}{rcl}
        d\Phi_1(\tau)&=&(A+B_c^1K_c(\tau))\Phi_1(\tau)d\tau\\
        &&\hspace{-7mm}+(E_cdW_1(\tau)+B_c^2K_c(\tau)dW_2(\tau))\Phi_1(\tau)
  \end{array}
  \end{equation}
  defined for $\tau\in[0,T_{max}]$ and $\Phi_1(0)=I$.
\item There exist matrix-valued functions $\tilde{S}:[0,T_{max}]\to\mathbb{S}^n$, $\tilde{S}(0)\succ0$, $U_c:[0,T_{max}]\to\mathbb{R}^{m_c\times n}$, a matrix $U_d\in\mathbb{R}^{m_d\times n}$ and a scalar $\eps>0$ such that the conditions
  \begin{equation}\label{eq:RDTz1}
  \begin{bmatrix}
    \Lambda(\tau) & \star & \star\\
   E_c\tilde{S}(\tau) & -\tilde{S}(\tau) & 0\\
    B_c^2U_c(\tau) & \star & -\tilde{S}(\tau)
  \end{bmatrix}\preceq0
\end{equation}
and
\begin{equation}\label{eq:RDTz2}
  \begin{bmatrix}
    -\tilde{S}(0)+\eps I & \star & \star & \star\\
    J\tilde{S}(0)+B_{d}^1U_d & -\tilde{S}(\theta) & \star & \star\\
    E_d\tilde{S}(0) & 0 & -\tilde{S}(\theta) & \star\\
    B_{d}^2U_d & 0 & 0 & -\tilde{S}(\theta)
  \end{bmatrix}\preceq0
\end{equation}
hold for all $\tau\in[0,T_{max}]$ and all $\theta\in[T_{min},T_{max}]$ where $\Lambda(\tau)=\dot{\tilde{S}}(\tau)+\He[A\tilde{S}(\tau)+B_{c}^1U_c(\tau)]$.
\end{enumerate}
Moreover, when one of the above statements holds, then the closed-loop system \eqref{eq:mainsyst}-\eqref{eq:SFranged} is mean-square asymptotically stable under ranged dwell-time $(T_{min},T_{max})$ and suitable controller gains can be computed from the conditions of statement (b) using the expressions
\begin{equation}\label{eq:controllzreskmds}
  K_c(\tau)=U_c(\tau)\tilde{S}(\tau)^{-1}\ \textnormal{and}\ K_d=U_d\tilde{S}(0)^{-1}.
\end{equation}
\end{theorem}
\begin{proof}
A Schur complement on the inequality \eqref{eq:RDTz1} and a congruence transformation with respect to $S(\tau)=\tilde{S}(\tau)^{-1}$ yields the inequality
\begin{equation}
\begin{array}{l}
    \hspace{-3mm}-S(\tau)+\He[S(\tau)(A+B_c^1K_c(\tau))]+E_c^TS(\tau)E_c\\
    \hspace{3cm}+K_c(\tau)^TB_c^{2T}S(\tau)B_c^2K_c(\tau)\preceq0.
\end{array}
\end{equation}
Similarly, \eqref{eq:RDTz2} can be shown to be equivalent to
\begin{equation}
\begin{array}{l}
    \hspace{-3mm}-S(0)+(J+B_d^1U_d)^TS(\theta)(J+B_d^1U_d)\\
    \hspace{1.5cm}+E_d^TS(\theta)E_d+K_d^TB_d^{2T}S(\tau)B_d^2K_d\prec0.
\end{array}
\end{equation}
Invoking now Theorem \ref{th:aperiodic}, Statement (b) proves the equivalence between the statements of the result.
\end{proof}

\subsection{Stabilization under minimum dwell-time}

Let us consider now the minimum dwell-time case. Since the dwell-time can be arbitrarily large in this setting, we propose to use the following state-feedback control law
\begin{equation}
  \begin{array}{rcl}
    u_c(t_k+\tau)&=&\left\{\begin{array}{l}
      K_c(\tau)x(t_k+\tau),\ \tau\in(0,\bar{T}]\\
      K_c(\bar{T})x(t_k+\tau),\ \tau\in(\bar{T},T_k]
    \end{array}\right.\\
    u_d(k)&=&K_dx(t_k)
  \end{array}
\end{equation}
where $K_c:[0,\bar{T}]\to\mathbb{R}^{m_c\times n}$ and $K_d\in\mathbb{R}^{m_d\times n}$ are the gains of the controllers that have to be determined. We then have the following result:
\begin{theorem}\label{th:minDTz}
The following statements are equivalent:
\begin{enumerate}
  \item There exists a matrix $P\in\mathbb{S}^n_{\succ0}$ such that the LMIs
  \begin{equation}
  \begin{array}{l}
     \hspace{-7mm} \E[(J+B_d^1K_d)^T\Psi_2(\bar{T})(J+B_d^1K_d)]-P\\
        \hfill+E_d^T\Psi_2(\bar{T})E_d+K_d^T(B_d^2)^T\Psi_2(\bar{T})B_d^2K_d\prec0
  \end{array}
  \end{equation}
  and
  \begin{equation}
    A^TP+PA+E_c^TPE_c\prec0
  \end{equation}
  hold where $\Psi_2(\theta)=\Phi_2(\theta)^TP\Phi_2(\theta)$ and
  \begin{equation}
  \begin{array}{rcl}
        d\Phi_2(\tau)&=&(A+B_c^1K_c(\tau))\Phi_2(\tau)d\tau\\
        &&\hspace{-7mm}+(E_cdW_1(t)+B_c^2K_c(\tau)dW_2(t))\Phi_2(\tau)
  \end{array}
  \end{equation}
   defined for $\tau\ge0$ and $\Phi_2(0)=I$.
\item There exist matrix-valued functions $\tilde{S}:[0,T_{max}]\to\mathbb{S}^n$, $\tilde{S}(\bar{T})\succ0$ , $U_c:[0,T_{max}]\to\mathbb{R}^{m_c\times n}$, a matrix $U_d\in\mathbb{R}^{m_d\times n}$ and a scalar $\eps>0$ such that the conditions
     \begin{equation}
\begin{bmatrix}
    \Lambda(\tau) & E_c\tilde{S}(\tau)+B_c^2U(\tau)\\
  \star & -\tilde{S}(\tau)
\end{bmatrix}\preceq0
\end{equation}
\begin{equation}
  \begin{bmatrix}
    \He[A\tilde{S}(\bar{T})+B_{c}^1U_c(\bar{T})] & E_c\tilde{S}(\bar{T})+B_c^2U(\bar{T})\\
    \star & -\tilde{S}(\bar{T})
  \end{bmatrix}\prec0
\end{equation}
and
\begin{equation}
  \begin{bmatrix}
 -\tilde{S}(0)+\eps I & J\tilde{S}(\bar{T})+B_d^1U_d & E_d\tilde{S}(\bar{T})+B_d^2U_d\\
 \star & -\tilde{S}(\bar{T}) & 0\\
 \star & \star & -\tilde{S}(\bar{T})
  \end{bmatrix}\preceq0
\end{equation}
hold for all $\tau\in[0,\bar{T}]$ where $\Lambda(\tau)=\dot{\tilde{S}}(\tau)+\He[A\tilde{S}(\tau)+B_c^1U(\tau)]$.
\end{enumerate}
Moreover, when this is the case, the closed-loop system \eqref{eq:mainsyst}-\eqref{eq:SFranged} is mean-square asymptotically stable under minimum dwell-time $\bar{T}$ and suitable controller gains can be computed from the conditions of the statement (b) using the expressions
\begin{equation}
  K_c(\tau)=U_c(\tau)\tilde{S}(\tau)^{-1}\quad\textnormal{and}\quad K_d=U_d\tilde{S}(0)^{-1}.
\end{equation}
\end{theorem}
\begin{proof}
The proof is based on the same manipulations as in the proof of Theorem \ref{th:aperiodicz}, with the difference that Theorem \ref{th:minDT} is invoked in place of Theorem \ref{th:aperiodic}.
\end{proof}

\subsection{Example}

Let us consider the system \eqref{eq:mainsyst} with the matrices
\begin{equation}\label{eq:ex:DTz}
\begin{array}{l}
    A=\begin{bmatrix}
    1 & 1\\
    1 & -2
  \end{bmatrix},B_c^1=\begin{bmatrix}
    4\\0
  \end{bmatrix}, B_c^2=\begin{bmatrix}
    1\\0
  \end{bmatrix}, J=\begin{bmatrix}
    3 & 1\\
    1 & 2
  \end{bmatrix},\\
  E_c=\begin{bmatrix}
    1 & 0\\
    1 & 2
  \end{bmatrix}, E_d=0.2\begin{bmatrix}
    1 & 0\\
    1 & -1
  \end{bmatrix}, B_d^1=\begin{bmatrix}
    1\\0
  \end{bmatrix}, B_d^2=\begin{bmatrix}
    0\\0.1
  \end{bmatrix}.
\end{array}
\end{equation}
Using then Theorem \ref{th:minDTz} with minimum dwell-time $\bar{T}=0.1$ and polynomials of order 1 and 2, we obtain the controller gains $  K_d=\begin{bmatrix}
    -3.9165&   -2.9751
  \end{bmatrix}$ and
\begin{equation*}
  K_c(\tau)=\dfrac{1}{\textnormal{den}(\tau)}\begin{bmatrix}
 -0.0299 \tau^2  -0.2251\tau+    0.1605\\
 0.0045\tau^2   -0.0167   \tau+ 0.1881 
  \end{bmatrix}^T
\end{equation*}
where $\textnormal{den}(\tau)= 0.0291\tau^2+  0.5439\tau -0.1255$. 
Choosing then $x_0=(2, -2)$, we then obtain the trajectory for $(\E[||x(t)||_2^2])^{1/2}$ depicted in Fig.~\ref{fig:minDTz} where we can observe the convergence to 0, emphasizing the mean-square asymptotic stability of the closed-loop system.


\begin{figure}[h]
  \centering
  \hspace{-3mm}\includegraphics[width=0.49\textwidth]{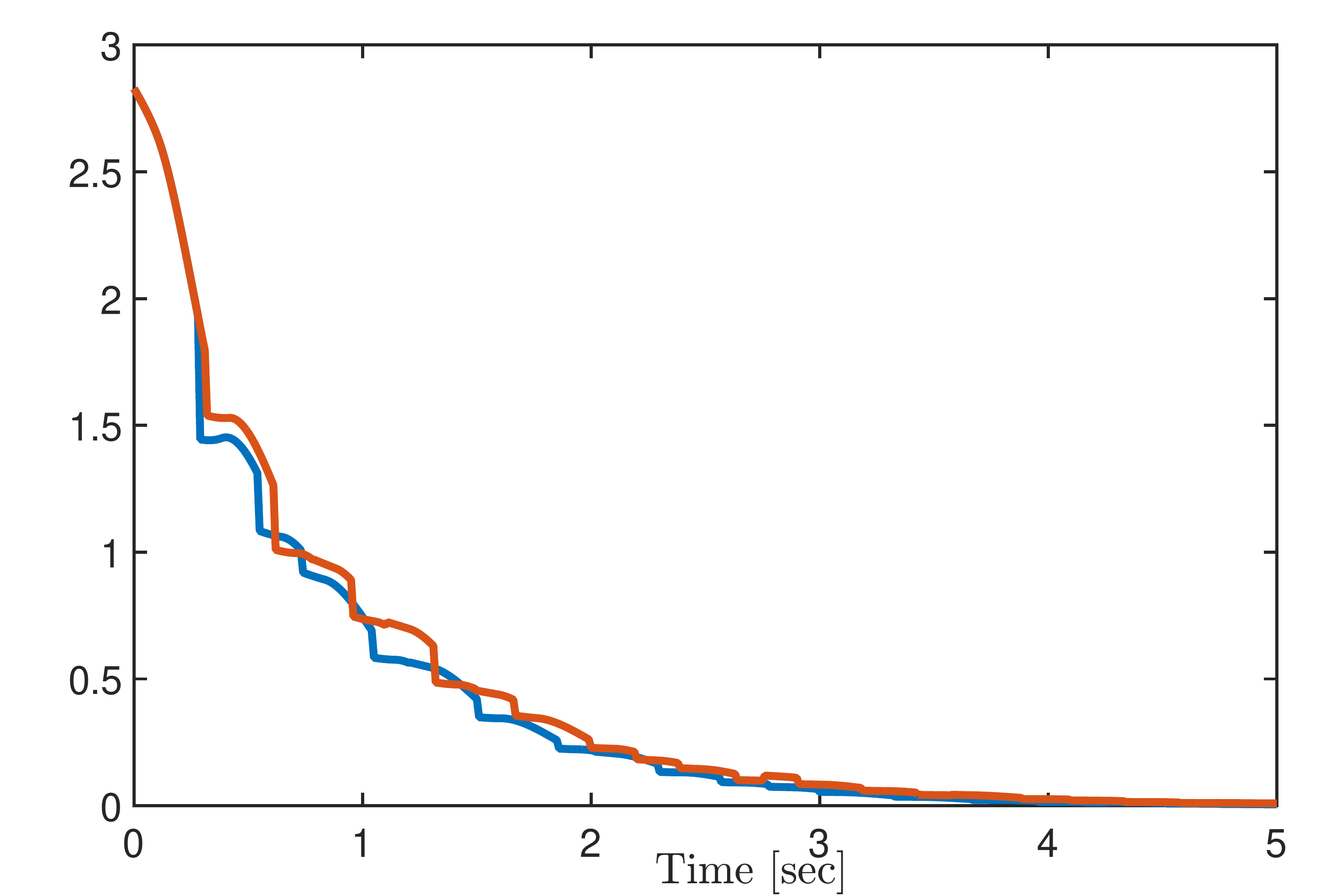}
  \caption{Evolution of $(\E[||x(t)||_2^2])^{1/2}$ along the trajectories of the system \eqref{eq:mainsyst}-\eqref{eq:ex:DTz} subject to two randomly generated impulse time sequences with minimum dwell-time $\bar T=0.1$.}\label{fig:minDTz}
\end{figure}

\section{Application to aperiodic sampled-data systems}\label{sec:SD}

As an application example, we utilize the ranged dwell-time result in order to derive a stabilization condition for linear aperiodic stochastic sampled-data systems represented in the impulsive form:
%
\begin{equation}\label{eq:SDsyst}
\begin{array}{lcl}
\displaystyle dx_{sd}(t)&=&\bar{A}x_{sd}(t)dt+\sum_{i=1}^2\bar{E}_ix_{sd}(t)dW_i(t)\\
  x_{sd}(t_k^+)&=&(J_0+\bar{B}K_d)x_{sd}(t_k)
  \end{array}
  \end{equation}
  where
  \begin{equation}
  \begin{array}{l}
        \bar A =\begin{bmatrix}
    A_{sd} & B_{sd}\\
    0 & 0
  \end{bmatrix},\ \bar{E}_1=\begin{bmatrix}
    E_{sd} & 0\\
    0 & 0
  \end{bmatrix},\ \bar{E}_2=\begin{bmatrix}
    0 & \alpha B_{sd}\\
    0 & 0
  \end{bmatrix}\\
    J_0=\begin{bmatrix}
    I & 0\\
    0 & 0
  \end{bmatrix},\ B=\begin{bmatrix}
    0\\
    I
  \end{bmatrix},\ K_d=\begin{bmatrix}
    K_d^1 & K_d^2
  \end{bmatrix}
  \end{array}
  \end{equation}
and $x_{sd}:=(x_{sd}^1,x_{sd}^2)$ where $x_{sd}^1\in\mathbb{R}^{n}$ is the state of the continuous-time system, $x_{sd}^2\in\mathbb{R}^{m}$ is the piecewise-constant state modeling the zero-order hold. As before $W_1,W_2\in\mathbb{R}$ are two independent zero mean Wiener processes. The parameter $\alpha>0$ is here to scale the amplitude of the noise on the control channel. The gain of the controller, denoted by $K_d$, is given by $K_d=\begin{bmatrix}
  K_d^1 & K_d^2
\end{bmatrix}$  where $K_d^1\in\mathbb{R}^{m\times n}$ and $K_d^2\in\mathbb{R}^{m\times m}$.

We then have the following result:
\begin{theorem}\label{th:rangeSD}
The following statements are equivalent:
\begin{enumerate}
  \item There exists a matrix $P\in\mathbb{S}^n_{\succ0}$ such that the LMI
  \begin{equation}
  \begin{array}{l}
    \E[(J_0+\bar{B}K_d)^T\Psi_{sd}(\theta)(J_0+\bar{B}K_d)]-P\prec0
  \end{array}
  \end{equation}
  holds for all $\theta\in[T_{min},T_{max}]$ where    $\Psi_{sd}(\theta):=\Phi_{sd}(\theta)^TP\Phi_{sd}(\theta)$ and
  \begin{equation}
  \begin{array}{rcl}
        d\Phi_{sd}(\tau)&=&\bar{A}\Phi_{sd}(\tau)d\tau+\bar{E}_1\Phi_{sd}(\theta)dW_1(t)\\
        &&+\bar{E}_2\Phi_{sd}(\theta)dW_2(t),\ \Phi_{sd}(0)=I
  \end{array}
  \end{equation}
  defined for $s\in[0,T_{max}]$.
\item There exist matrix-valued functions $\tilde{S}:[0,T_{max}]\to\mathbb{S}^{n+m}$, a matrix $U_d\in\mathbb{R}^{m\times (n+m)}$ and a scalar $\eps>0$ such that the conditions $\tilde{S}(0)\succ0$ and
  \begin{equation}
  \begin{bmatrix}
    \dot{\tilde{S}}(\tau)+\He[A\tilde{S}(\tau)] & \tilde{S}(\tau)E_1^T & \tilde{S}(\tau)E_2^T\\
    \star & -\tilde{S}(\tau) & 0\\
    \star & \star & -\tilde{S}(\tau)
  \end{bmatrix}\preceq0
\end{equation}
\begin{equation}
  \begin{bmatrix}
    -\tilde{S}(0)+\eps I & (J\tilde{S}(0)+B_{d}U_d)^T\\
    \star & -\tilde{S}(\theta)
  \end{bmatrix}\preceq0
\end{equation}
hold for all $\tau\in[0,T_{max}]$ and all $\theta\in[T_{min},T_{max}]$.
\end{enumerate}
Moreover, when one of the above statements holds, then the (closed-loop) sampled-data system \eqref{eq:SDsyst} is mean-square asymptotically stable under ranged dwell-time $T_k\in[T_{min},T_{max}]$ and a suitable controller gain can be computed from the expression $ K_d=U_d\tilde{S}(0)^{-1}$.
\end{theorem}

We now illustrate the above result by a simple example:
\begin{example}
  Let us consider the sampled-data system \eqref{eq:SDsyst} with $\alpha=0.1$ and the matrices
  \begin{equation}\label{eq:matSD}
    A_{sd}=\begin{bmatrix}
        0 & 1\\
        0 & -1
    \end{bmatrix},\ B_{sd}=\begin{bmatrix}
    0\\1
    \end{bmatrix}\ \textnormal{and}\ E_{sd}=\begin{bmatrix}
      0 & 0\\
      0 & 0.1
    \end{bmatrix}.
  \end{equation}
  Applying then Theorem \ref{th:rangeSD} with polynomials of order 2, we get the results gathered in Table \ref{tab:stabzSD}. Simulations results are depicted in Fig.~\ref{fig:SD} where we can see that the designed controllers effectively render the closed-loop system mean-square asymptotically stable under the considered ranged dwell-time conditions. We can also observe that the controllers give rise to similar performance in terms of the rate of convergence of $\E[||x(t)||_2^2]$ to 0 despite having different time-scales for the control input update.

  \begin{table}[H]
    \centering
    \caption{Various controller gains for the sampled-data system \eqref{eq:SDsyst}-\eqref{eq:matSD} obtained with Theorem \ref{th:rangeSD} using polynomials of order 2.}\label{tab:stabzSD}
    \begin{tabular}{|c|c||c|}
    \hline
      $T_{min}$ & $T_{max}$ & $K_d$\\
      \hline
      \hline
      \multirow{3}{*}{$0.001$} & 0.1 &$ (-0.4069, -0.1734, -0.0045)$\\
       & 0.5 & $(-0.4421  , -0.2137  , -0.0215    )$\\
       & 1 &$ ( -0.3410 ,  -0.1332  ,  0.0036 )$\\
       \hline
       \multirow{3}{*}{$1$} & 5 & $( -0.1977,   -0.1931  ,  0.0014     )$\\
       & 10 & $(   -0.1053  , -0.1061  ,  0.0011     )$\\
       &20 & $(     -0.0583 ,  -0.0559 ,  -0.0003   )$\\
       \hline
    \end{tabular}
  \end{table}

\begin{figure}[h]
  \centering
   \hspace{-4mm}\includegraphics[width=0.49\textwidth]{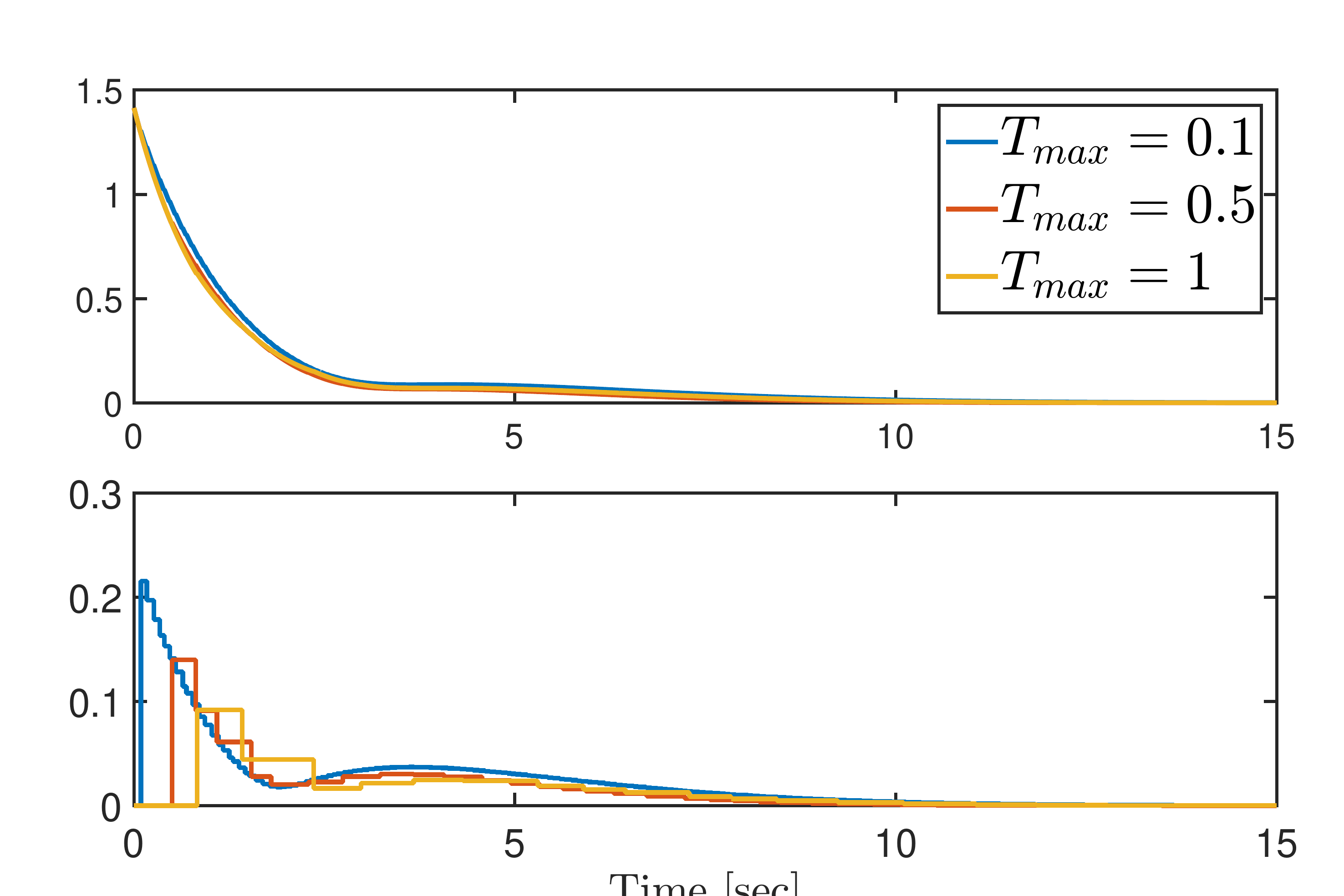}
  \caption{Evolution of $\E[||x(t)||_2^2]^{1/2}$ (top) and $\E[||u(t)||_2^2]^{1/2}$ (bottom) along the trajectories of the system  \eqref{eq:SDsyst}-\eqref{eq:matSD} with the controllers designed for $T_{min}=0.001$ and several values for $T_{max}$. The sequences of impulse times satisfying the ranged dwell-time condition are randomly generated.}\label{fig:SD}
\end{figure}
\end{example}

\section{Conclusion}

Dwell-time stability and stabilization conditions have been obtained for stochastic impulsive systems and expressed as infinite-dimensional LMI problems that can be solved using discretization or sum of squares techniques. The approach has been shown to include switched systems and sampled-data systems as particular cases, making then the approach quite general.

Possible future works include the design of dynamic output feedback controllers, observer and filters. The approach can also be extended to performance characterization using induced-norms such as the induced $L_2$-norm and the induced $L_2$-$L_\infty$-norm. Other possible extensions would concern other concepts of dwell-times \cite{Zhang:15a} and the consideration of controller/mode mismatch due to decision delays \cite{Zhang:16b}.


\end{document}